\documentclass{amsart}
\usepackage{amsmath, amsthm, amssymb, amscd, mathrsfs, eucal, epsfig}
\usepackage{pinlabel}

\begin{document}

\newtheorem{theorem}{Theorem}[section]
\newtheorem{lemma}[theorem]{Lemma}
\newtheorem{proposition}[theorem]{Proposition}
\newtheorem{corollary}[theorem]{Corollary}
\newtheorem{conjecture}[theorem]{Conjecture}
\newtheorem{question}[theorem]{Question}
\newtheorem{problem}[theorem]{Problem}
\newtheorem*{claim}{Claim}
\newtheorem*{criterion}{Criterion}
\newtheorem*{main_rat_thm}{Rationality Theorem}

\theoremstyle{definition}
\newtheorem{definition}[theorem]{Definition}
\newtheorem{construction}[theorem]{Construction}
\newtheorem{notation}[theorem]{Notation}

\theoremstyle{remark}
\newtheorem{remark}[theorem]{Remark}
\newtheorem{example}[theorem]{Example}

\numberwithin{equation}{subsection}

\def\Z{\mathbb Z}
\def\R{\mathbb R}
\def\Q{\mathbb Q}
\def\D{\mathcal D}
\def\E{\mathcal E}
\def\RR{\mathcal R}
\def\P{\mathcal P}
\def\F{\mathcal F}

\def\cl{\textnormal{cl}}
\def\scl{\textnormal{scl}}
\def\homeo{\textnormal{Homeo}}
\def\rot{\textnormal{rot}}

\def\Id{\textnormal{Id}}
\def\SL{\textnormal{SL}}
\def\PSL{\textnormal{PSL}}
\def\length{\textnormal{length}}
\def\fill{\textnormal{fill}}
\def\rank{\textnormal{rank}}
\def\til{\widetilde}

\title{Stable commutator length is rational in free groups}
\author{Danny Calegari}
\address{Department of Mathematics \\ Caltech \\
Pasadena CA, 91125}
\email{dannyc@its.caltech.edu}

\date{3/13/2009, Version 0.10}
\dedicatory{Dedicated to Shigenori Matsumoto on the occasion of his 60th birthday.}
\begin{abstract}
For any group, there is a natural (pseudo-)norm on the vector space $B_1^H$ of real
homogenized (group) $1$-boundaries, called the
{\em stable commutator length} norm. This norm is closely related to, and can be
thought of as a relative version of, the Gromov
(pseudo)-norm on (ordinary) homology. We show that for a free group,
the unit ball of this pseudo-norm is a rational polyhedron.

It follows that stable commutator length in free groups takes on only rational
values. Moreover every element of the commutator subgroup
of a free group rationally bounds an injective map of a surface group.

The proof of these facts yields an algorithm to compute stable
commutator length in free groups. Using this algorithm, we answer a well-known
question of Bavard in the negative, constructing explicit examples of elements
in free groups whose stable commutator length is not a half-integer.
\end{abstract}

\maketitle

\section{Introduction}
Stable commutator length is a numerical invariant of elements in the commutator subgroup
of a group. It is intimately related to (two-dimensional) bounded cohomology, and appears
in many areas of low-dimensional topology and dynamics, from the Milnor-Wood inequality,
to the $11/8$ conjecture. However, although a great deal of work has gone into
estimating or bounding stable commutator length in many contexts, there are very few known
nontrivial examples of finitely presented groups in which it can be calculated exactly,
and virtually no cases where the range of stable commutator length on a given group 
can be understood arithmetically.

The most significant results of this paper are as follows:

\begin{enumerate}
\item{We show that stable commutator length in free groups takes on only rational values,
and give an explicit algorithm to compute the value on any given element. This is the
first example of a group with infinite dimensional second bounded cohomology group $H^2_b$
in which stable commutator length can be calculated exactly.}
\item{We show how to extend stable commutator length to a (pseudo)-norm on the vector
space $B_1^H$ of homogenized real (group) $1$-cochains that are boundaries of $2$-cochains. 
In the case of a free group, this is a genuine norm. We show that the intersection of
the unit ball in this norm with any finite dimensional rational subspace of
$B_1^H$ of a free group is a finite-sided rational polyhedron. This invites
comparison with the Thurston norm on the $2$-dimensional homology of
a $3$-manifold \cite{Thurston_norm}, although the relationship between the two
cases is subtle and deserves further investigation.}
\item{We give examples of explicit elements in the commutator subgroup of $F_2$ (the
free group of rank $2$) for which the stable commutator length is not in $\frac 1 2 \Z$. This
answers in the negative a well-known question of Bavard \cite{Bavard}.}
\end{enumerate}

We now elaborate on these points in turn.

\medskip

Let $G$ be a group. For $g \in [G,G]$, the {\em commutator length} of $g$, denoted
$\cl(g)$, is the smallest number of commutators in $G$ whose product is equal to $g$.
The {\em stable commutator length} of $g$, denoted $\scl(g)$, is the limit of $\cl(g^n)/n$
as $n \to \infty$. In geometric language, (see e.g.\ Gromov \cite{Gromov_asymptotic})
$\cl$ is sometimes called ``filling genus'',
and $\scl$ is called ``stable filling genus''. This quantity is intimately related,
by Bavard duality and an exact sequence (see Theorem~\ref{duality_theorem} and 
Proposition~\ref{exact_sequence_prop}) to the second bounded cohomology $H^2_b$ of $G$, 
with its Banach norm. Despite a considerable amount of research, there are very
few known examples of groups $G$ in which $\scl$ can be calculated exactly (except when it
vanishes identically). This is partly because the groups $H^2_b$, when nontrivial,
tend to be very large in general: when $G$
is word-hyperbolic, $H^2_b$ is not merely infinite dimensional, but is not even separable
as a Banach space. Calculating $\scl$ is tantamount to solving an extremal problem
in $H^2_b$. (Technically, one solves the extremal problem in the space of {\em homogeneous
quasimorphisms} $Q$, see Definition~\ref{quasimorphism_definition}. The spaces $Q$
and $H^2_b$ are both Banach spaces, and are related by the coboundary $\delta$, 
which is Fredholm when $G$ is finitely presented.) 

Gromov (\cite{Gromov_asymptotic} 6.$C_2$) asked whether $\scl$ is always rational in
finitely presented groups. The answer to Gromov's question is known to be {\em no}:
Dongping Zhuang gave the first examples in \cite{Zhuang}. These examples occur
in generalized Stein-Thompson groups of PL homeomorphisms of the circle, where one
can show that $H^2_b$ is actually finite dimensional, and everything can be calculated
explicitly. In this paper we show that {\em $\scl$ is rational} in free
groups (and some closely related groups), and moreover we give an explicit algorithm
to compute the value of $\scl$ on any element. 

\medskip

If $g_1,g_2,\cdots,g_m$ are elements in $G$, define $\cl(g_1 + \cdots + g_m)$ to be
the smallest number of commutators in $G$ whose product is equal to the product of
conjugates of the $g_i$. Let $\scl(g_1 + \cdots + g_m)$ denote
the limit of $\cl(g_1^n + \cdots + g_m^n)/n$ as $n \to \infty$. This function can
be extended by linearity and continuity in a unique way to a pseudo-norm on $B_1$, the
vector space of real group $1$-chains on $G$ that are in the image of the boundary
map $\partial:C_2 \to C_1$. This function vanishes identically on the subspace $H$ of $B_1$
spanned by terms of the form $g^n - ng$ and $g - hgh^{-1}$ for $g,h \in G$ and $n \in \Z$, and
descends to a pseudo-norm on the quotient $B_1/H$, or $B_1^H$ for short. When $G$ is hyperbolic,
$\scl$ is a genuine norm on $B_1^H$. We show that in a free group, this $\scl$ norm is {\em piecewise
rational linear} (denoted PQL) on finite dimensional rational subspaces of $B_1^H$. So
for any finite set of elements $g_1,g_2,\cdots,g_m \in G$, there is a uniform upper
bound on the denominators of the values of $\scl$ on integral linear chains $\sum_i n_ig_i$ in $B_1^H$.

One should compare the $\scl$ norm with the Gromov-Thurston norm
\cite{Thurston_norm}, which is a norm on $H_2$ of an irreducible, atoroidal
$3$-manifold, and whose most significant feature is that it is a piecewise rational linear
function. In Thurston's definition (in which one restricts to embedded surfaces) this is
straightforward to show. In Gromov's definition (in terms of chains,
or immersed surfaces) this is a very deep theorem, whose proof depends on the full power of Gabai's
theory of sutured hierarchies \cite{Gabai_foliations}, and taut foliations.
In fact, it is reasonable to think of the $\scl$ norm as a relative  
Gromov-Thurston norm, with Gromov's definition. Our proof of rationality is conceptually
close in some ways to an argument due to Oertel \cite{Oertel} insofar as both proofs reduce the
problem of calculating the norm to a linear programming problem
in the vector space of weights carried by a finite constructible branched surface.
However, there are crucial differences between the two cases.
In Oertel's case, the branched surface might have complicated branch locus, but it
comes with an embedding in a $3$-manifold. In our case, the branch locus is simple, but
the branched surface is merely {\em immersed} in a $3$-manifold. It is intriguing to try to
find a natural generalization of both theories.

\medskip

In his seminal paper \cite{Bavard} on stable commutator length, Bavard
asked whether stable commutator length in free groups takes values in $\frac 1 2 \Z$. 
There were several pieces of direct and indirect evidence for this conjecture. Firstly,
where certain (geometric or homological) methods for estimating stable commutator length
in free groups give exact answers, these answers in every case confirm Bavard's
guess. Secondly, in the (analogous) context of $3$-manifold topology, one knows that
the Gromov norm of an integral $2$-dimensional homology class is in $2\Z$ (the factor
of $4$ arises because Gromov norm counts triangles, whereas stable commutator norm
counts handles). It was generally felt that Bavard's conjecture was eminently plausible,
and it is therefore surprising that our algorithm produces many elements whose stable
commutator length is not in $\frac 1 2 \Z$. In fact experiments suggest that arbitrarily
large denominators occur, with arbitrary prime factors. In view of these examples,
the fact that stable commutator length is rational in free groups is seen to be a more delicate and
subtle fact than one might have imagined, and stable commutator length to be a richer
invariant than previous work has suggested.

\medskip

The organization of this paper is as follows. In \S~\ref{background_section} we state
definitions and sketch proofs of background results which pertain to stable commutator
length in groups in general. In \S~\ref{free_group_section} we specialize to the case of free groups.
The purpose of this section is to state and prove the ``Rationality Theorem'',
whose precise statement is the following:
\begin{main_rat_thm}
Let $F$ be a free group.
\begin{enumerate}
\item{$\scl(g) \in \Q$ for all $g \in [F,F]$.}
\item{Every $g  \in [F,F]$ rationally bounds an extremal surface (in fact, every rational chain $C$ in $B_1^H$ rationally
bounds an extremal surface)}
\item{The function $\scl$ is piecewise rational linear on $B_1^H$.}
\item{There is an algorithm to calculate $\scl$ on any finite dimensional rational subspace of
$B_1^H$.}
\end{enumerate}
\end{main_rat_thm}
Similar rationality results hold for stable commutator length in virtually free groups,
and fundamental groups of noncompact Seifert-fibered $3$-manifolds.
Finally, in \S~\ref{algorithm_section} we explicitly describe an algorithm for computing the stable
commutator length in free groups, and discuss a simple
example that answers Bavard's question in the negative.

\section{Background}\label{background_section}

For the convenience of the reader, we collect here some basic definitions and properties
that will be used in subsequent sections. As general background, see \cite{Bavard}, 
\cite{Brooks}, \cite{Calegari_scl} and \cite{Gromov_bounded}. 
Note that the reference \cite{Calegari_scl}, although
the most detailed, complete and relevant to the material in this paper, 
is an unfinished manuscript (which is readily available online)
and therefore we have tried to refer to this manuscript by section number (which one can expect
to be reasonably stable) rather than by page number.

\subsection{Stable commutator length}
In this section we give the definitions and basic properties of stable commutator length
in groups. This is a numerical invariant of elements in the commutator subgroup of
a given group which is {\em universal} for certain kinds of extremal problems. For
background or proofs, see \cite{Bavard} or \cite{Calegari_scl}.

\begin{definition}\label{scl_definition}
Let $G$ be a group. For $g \in [G,G]$ the {\em commutator length} of $g$, denoted
$\cl(g)$, is the smallest number of commutators in $G$ whose product is equal to $g$.
The {\em stable commutator length}, denoted $\scl(g)$, is the following limit
$$\scl(g): = \lim_{n \to \infty} \frac {\cl(g^n)} n$$
\end{definition}

Note that the function $\cl(g^n)/n$ is subadditive, so the limit in 
Definition~\ref{scl_definition} exists. Notice further that
$\cl$ and $\scl$ are class functions, and that they are monotone non-increasing
under homomorphisms between groups. If we need to emphasize that $\cl$ or $\scl$
is being calculated in a fixed group, we will use subscripts; hence $\cl_G(g)$ and
$\scl_G(g)$. 

\begin{remark}
We sometimes extend $\cl$ and $\scl$ to all of $G$ by defining $\cl(g) = \infty$ if $g$ is not
in $[G,G]$, and replacing $\lim$ by $\liminf$ in the definition of $\scl$. 
Notice that $\scl(g) < \infty$ if and only if some power of $g$ is in
$[G,G]$.
\end{remark}

The functions $\cl$ and $\scl$ can be
interpreted geometrically. Let $X$ be a connected CW complex with $\pi_1(X) = G$,
and let $\gamma$ be a loop in $X$ whose free homotopy class represents the conjugacy
class of $g$. Then $\cl(g) \le n$ if and only if there exists an orientable surface
$S$ of genus $n$ with one boundary component, and a map $f:S \to X$
taking $\partial S$ to the free homotopy class of $\gamma$.

\begin{remark}
In order to be able to speak interchangeably about loops $\gamma$ in spaces $X$ as above and
their images, we assume in the sequel that all spaces $X$ are such that every free homotopy
class of loop can be realized by an embedded circle. This can be achieved, for an arbitrary
homotopy type of CW complex $X$, by multiplying by a sufficiently high dimensional cube.
\end{remark}

Genus is not multiplicative under finite covers, but Euler characteristic is.
So when we stabilize $\cl$, the relevant geometric quantity to keep track of is
derived from Euler characteristic.

\begin{notation}
Let $S$ be a compact, connected, oriented surface. Then set
$$\chi^-(S) = \min(0,\chi(S))$$
Extend $\chi^-$ additively to compact, oriented (but not necessarily connected)
surfaces $S$, so that
$$\chi^-(S) = \sum_i \chi^-(S_i)$$
where $S_i$ ranges over the components of $S$ (cf.\ \cite{Thurston_norm}).
\end{notation}

\begin{notation}
Let $S$ be a compact, connected, oriented surface. Let $X$ be a topological space, and
$\gamma:S^1 \to X$ a continuous loop. Let $f:S \to X$ be such that there is a commutative diagram
$$ \begin{CD}
\partial S @>i>> S \\
@V\partial fVV  @VVfV \\
S^1 @>\gamma>> X 
\end{CD}$$
where $i:\partial S \to S$ is the inclusion map, and define $n(S)$ by the identity
$\partial f_*[\partial S] =  n(S)[S^1]$ in $H_1$. If the (oriented) components of $\partial S$ are
denoted $\partial_i$, then $n(S)$ is the sum of the degrees of the maps $\partial f:\partial_i \to S^1$.
Informally, $n(S)$ is the degree with which $\partial S$ wraps around the loop $\gamma$.
\end{notation}

If $n(S)$ is nonzero, one says that the surface $S$ {\em rationally bounds $\gamma$}.
Strictly speaking, $n(S)$ depends on $f$ and not just on $S$, but we suppress this in our
notation. 

With these definitions, one can give a geometric interpretation of $\scl$.

\begin{lemma}\label{euler_lemma}
Let $X$ be a connected CW complex with $\pi_1(X) = G$ and let $\gamma$ be a loop in $X$
in the free homotopy class corresponding to the conjugacy class of $g$. Then
$$\scl(g) = \inf_S \frac {-\chi^-(S)} {2|n(S)|}$$
where the infimum is taken over all maps $f:S \to X$ wrapping $\partial S$ around
$\gamma$ with any degree $n(S)$.
\end{lemma}
\begin{proof}
An inequality in one direction
can be obtained by restricting the class of admissible $S$ to those that are connected
with exactly one boundary component. To obtain the inequality in the other direction,
first observe that components without boundary can be thrown away without
increasing $-\chi^-$. Passing to a cover multiplies both $-\chi^-$ and $n$ by the same factor.
Moreover an orientable surface with $p$ boundary components admits a cover (in fact, a cyclic
cover) of degree $m$ which also has $p$ boundary components, providing $m$ and $p-1$ are coprime.
After passing to such a cover with $m$ very large, multiple boundary components can be tubed together with
$1$-handles (whose image in $X$ can be taken to be a point), 
increasing $-\chi^-$ by a term which is arbitrarily small compared to $n$, thereby proving the
theorem.
\end{proof}

By changing the orientation on $S$ if necessary, we may always take $n(S)$ to be positive. In the sequel
we therefore adhere to the convention that $n(S)$ is positive unless we explicitly say otherwise.
On the other hand, even if $n(S)$ is positive, if $S$ has more than one boundary component, 
some components might map to $\gamma$ with positive degree, and others with negative
degree. Say further that a surface $S$ is {\em monotone} if the degree of every component $\partial_i \to \gamma$ is
positive. The following lemma shows that for the purposes of computing $\scl$, one can restrict
attention to monotone surfaces.

\begin{lemma}\label{monotone_surface_lemma}
Let $f:S \to X$ be a connected surface with $\chi(S)<0$ that rationally bounds $\gamma$. Then
there is another surface $f':S' \to X$ which is monotone, and satisfies $-\chi^-(S')/2n(S') \le -\chi^-(S)/2n(S)$.
\end{lemma}
\begin{proof}
Each component $\partial_i$ of $\partial S$ maps to $\gamma$ with some degree $n_i \in \Z$, where 
$\sum_i n_i = n(S)$. If some $n_i$ is zero, the image $f(\partial_i)$ is homotopically trivial in $X$, so we may
reduce $-\chi^-(S)$ without affecting $n(S)$ by compressing $\partial_i$. So without loss of generality, assume
that no $n_i$ is zero.

Since $\chi(S)$ is negative, there is a finite cover of $S$ with positive genus. If $S$ is a connected
surface with positive genus and negative Euler characteristic, there is a connected 
degree $2$ cover of $S$ such that each boundary component in $S$ has exactly two preimages in the cover.
Hence, after passing to a finite cover if necessary (which does not affect the ratio of $-\chi^-$ to
$n(\cdot)$) we can assume that the boundary components $\partial_i$ of $S$ come in {\em pairs}
with equal degrees $n_i$.

Let $N$ be the least common multiple of the $|n_i|$. Define a function $\phi$ from the set of
boundary components of $S$ to $\Z/N\Z$ as follows. Divide the set of components into pairs
$\partial_i,\partial_j$ for which $n_i = n_j$, and define $\phi(\partial_i) = n_i$ and
$\phi(\partial_j) = -n_i$. Then $\sum_i \phi(\partial_i) = 0$, so $\phi$ extends to a surjective
homomorphism from $\pi_1(S)$ to $\Z/N\Z$. If $S''$ is the cover associated to the kernel, then
each component of $\partial S''$ maps to $\gamma$ with degree $\pm N$. Pairs of components whose
degrees have opposite sign can be glued up (which does not affect $-\chi^-$ or $n(\cdot)$) until all
remaining components have degrees with the same (positive) sign. The resulting surface $S'$
satisfies the conclusion of the lemma.
\end{proof}

For more details, see \cite{Calegari_scl}, \S~2.1.

\subsection{Extremal surfaces}

\begin{definition}
A map $f:S \to X$ rationally bounding $\gamma$ is {\em extremal} if $S$ has no disks or closed
components, and there is an
equality $\scl(g) = -\chi^-(S)/2n(S)$.
\end{definition}

Notice that $\scl(g)$ must be rational for an extremal surface to exist.

\begin{lemma}\label{extremal_lemma}
An extremal surface is $\pi_1$-injective.
\end{lemma}
\begin{proof}
Let $f:S \to X$ be extremal. Suppose there is some essential immersed loop $\alpha$ in $S$ for which
$f(\alpha)$ is null-homotopic. Since surface groups are LERF (see \cite{Scott}) there is a finite cover
$\widehat{S}$ of $S$ to which $\alpha$ lifts as an embedded loop. Let $\widehat{f}:\widehat{S} \to X$
lift the map $f$ (i.e.\ $\widehat{f}$ is the composition of $f$ with the
covering projection $\widehat{S} \to S$). Note that $-\chi^-(\widehat{S})/2n(\widehat{S}) = -\chi^-(S)/2n(S)$.

Let $\widehat{\alpha}$ denote an embedded preimage of 
$\alpha \subset S$ in $\widehat{S}$. Since $\widehat{\alpha}$ is nullhomotopic under $\widehat{f}$,
we can surger $\widehat{S}$ along $\widehat{\alpha}$ 
to produce a surface $S'$ with $-\chi^-(S') < -\chi^-(\widehat{S})$
but with $n(S') = n(\widehat{S})$. But this contradicts the hypothesis that $S$ is extremal.
\end{proof}

Lemma~\ref{monotone_surface_lemma} shows that if there is an extremal surface for $\gamma$, there
is a monotone extremal surface.

\subsection{Quasimorphisms}

A brief discussion of quasimorphisms, though not strictly logically necessary for the results of
this paper, nevertheless provides some useful context and explains an important connection
with the theory of bounded cohomology.

\begin{definition}\label{quasimorphism_definition}
Let $G$ be a group. A {\em quasimorphism} on $G$ is a function $\phi:G \to \R$ for
which there exists some least non-negative constant $D(\phi)$ called the {\em defect},
so that the following inequality holds
$$|\phi(g) + \phi(h) - \phi(gh)| \le D(\phi)$$
for all $g,h \in G$. A quasimorphism is {\em homogeneous} if $\phi(g^n) = n\phi(g)$ for
all $g \in G$ and all $n \in \Z$.
\end{definition}
In words, a quasimorphism on a group is a homomorphism up to a bounded error. A quasimorphism
is a genuine homomorphism if and only if the defect is zero.

The set of quasimorphisms (resp. homogeneous quasimorphisms) on $G$ admits the structure of
a real vector space. Denote the vector space of quasimorphisms on $G$ by $\widehat{Q}(G)$,
and the vector space of homogeneous quasimorphisms by $Q(G)$.

\begin{proposition}[Bavard \cite{Bavard}, Prop. 3.3.1]\label{exact_sequence_prop}
Let $G$ be a group. There is an exact sequence
$$0 \to H^1(G;\R) \to Q(G) \xrightarrow{\delta} H^2_b(G;\R) \to H^2(G;\R)$$
where $H^*_b$ denotes {\em bounded cohomology} (with real coefficients), and
$\delta$ denotes the coboundary on group $1$-cochains.
\end{proposition}
See \cite{Gromov_bounded} for an introduction to bounded cohomology. 
Note that when $H^1$ and $H^2$ are finite dimensional (as is the case when $G$
is finitely presented) then $\delta$ is Fredholm (with respect to natural Banach norms
on $Q/H^1$ and $H^2_b$).

There is a kind of duality, called {\em Bavard duality}, 
relating commutator length and quasimorphisms.
The most concise statement of this duality is the following:
\begin{theorem}[Bavard's Duality Theorem \cite{Bavard}, p. 111]\label{duality_theorem}
Let $G$ be a group. For any $g \in [G,G]$ there is an equality
$$\scl(g) = \frac 1 2 \sup_{\phi \in Q(G)} \frac {\phi(g)} {D(\phi)}$$
\end{theorem}
Note that one should restrict attention to $\phi \in Q(G) - H^1(G)$ since if $\phi$ is
a homomorphism, then both $\phi(g) = 0$ for $g \in [G,G]$, and $D(\phi)=0$.

\subsection{Stable commutator length as a norm}

The functions $\cl$ and $\scl$ can be extended to finite sums as follows.
\begin{definition}\label{general_scl_definition}
Let $G$ be a group. Let $g_1, \cdots,g_m$ be elements in $G$ (not necessarily distinct).
Define 
$$\cl(g_1 + g_2 + \cdots + g_m) = 
\inf_{h_1,\cdots,h_{m-1} \in G} \cl(g_1h_1g_2h_1^{-1}h_2g_3h_2^{-1} \cdots h_{m-1}g_mh_{m-1}^{-1})$$
and define
$$\scl(g_1 + g_2 + \cdots + g_m) = \lim_{n \to \infty} \frac {\cl(g_1^n + \cdots + g_m^n)} n$$
\end{definition}
Note that $\cl$ and $\scl$ depend only on the individual conjugacy classes of the summands, and
are commutative in their arguments.
Geometrically, if $X$ is a CW complex with $\pi_1(X) = G$ and $\gamma_1,\cdots,\gamma_m$ are loops
representing the conjugacy classes of $g_1,\cdots,g_m$ respectively, then $\cl(\sum g_i)$ is the smallest
genus surface $S$ with $m$ boundary components $\partial_i$ for which there is a map $f:S \to X$ 
wrapping each $\partial_i$ around $\gamma_i$.
It is worth remarking that the function $\cl_n:=\cl(\sum g_i^n)$ is not subadditive, but that the
``corrected'' function $\cl_n + (m-1)$ is subadditive, and therefore the limit exists in
Definition~\ref{general_scl_definition}, providing $\cl$ is not infinite. 

\begin{notation}
Let $S$ be a compact, connected, oriented surface. Let $X$ be a topological space, and
$\gamma_i:S^1 \to X$ for $1 \le i \le m$ continuous loops. Let $f:S \to X$ be such that there is a commutative diagram
$$ \begin{CD}
\partial S @>i>> S \\
@V\partial fVV  @VVfV \\
\coprod_i S^1 @>\coprod_i \gamma_i>> X 
\end{CD}$$
where $i:\partial S \to S$ is the inclusion map. Suppose there is an integer
$n(S)$ so that
$\partial f_*[\partial S] =  n(S)[\coprod_i S^1]$ in $H_1$. Then say $f:S \to X$ is
{\em admissible}. Informally, $n(S)$ is the common degree with which $\partial S$ wraps 
around each loop $\gamma_i$.
\end{notation}

With this notation, the generalization of Lemma~\ref{euler_lemma} to arbitrary sums is as follows:
\begin{lemma}\label{general_euler_lemma}
Let $X$ be a connected CW complex with $\pi_1(X)=G$. Further, let $\gamma_1,\cdots,\gamma_m$ be loops in
$X$ in free homotopy classes corresponding to conjugacy classes $g_1,\cdots,g_m$. Then
$$\scl(\sum_i g_i) = \inf_S \frac {-\chi^-(S)} {2|n(S)|}$$
where the infimum is taken over all admissible maps $f:S \to X$ wrapping $\partial S$ around each $\gamma_i$ with
degree $n(S)$.
\end{lemma}
The proof is almost identical to that of Lemma~\ref{euler_lemma}. Moreover, one may restrict attention
to monotone admissible maps, by the argument of Lemma~\ref{monotone_surface_lemma}.
For details, see \cite{Calegari_scl}, \S~2.6.1.

The function $\scl$ as above can be extended to integral group $1$-chains.
From Lemma~\ref{general_euler_lemma} follow equalities
$$\scl(g + g^{-1} + \sum g_i) = \scl(\sum g_i)$$
and
$$\scl(g^n + \sum g_i) = \scl(\underbrace{g + \cdots + g}_n + \sum g_i)$$
valid for any $g,g_i$ and any non-negative integer $n$. Hence one may define
$$\scl(\sum n_ig_i):=\scl(\sum g_i^{n_i})$$
for any integers $n_i$ (not necessarily non-negative) and elements $g_i \in G$ and observe that the result
is well-defined on integral group $1$-chains, and is subadditive under addition of chains.
Consequently, $\scl$ can be extended to rational chains by linearity, and to real chains by
continuity. See \cite{Calegari_scl}, \S~2.6.1.

Denote the vector space of real (group) $1$-chains on $G$ by $C_1(G)$ and $1$-boundaries by $B_1(G)$
(or just $C_1$ and $B_1$ if $G$ is understood).
The expression $\scl(\sum t_i g_i)$ is finite if and only if $\sum t_i g_i \in B_1$.
Bavard duality holds in the broader context of arbitrary real $1$-boundaries, 
and with essentially the same proof. The statement is:
\begin{theorem}[Generalized Bavard Duality \cite{Calegari_scl} \S~2.6.2]\label{general_duality_theorem}
Let $G$ be a group. For any finite set of elements $g_i \in G$ and numbers $t_i \in \R$ for which
$\sum t_i g_i \in B_1$ there is an equality
$$\scl(\sum t_i g_i) = \frac 1 2 \sup_{\phi \in Q(G)} \frac {\sum t_i \phi(g_i)} {D(\phi)}$$
\end{theorem}
Any homogeneous quasimorphism vanishes identically on any chain of the form $g^n - ng$ or
$g - hgh^{-1}$ for $g,h\in G$ and $n \in \Z$. Define $H$ to be the subspace of $B_1$ spanned by
such chains. Then $\scl$ descends to a pseudo-norm on $B_1(G)/H$ (hereafter denoted $B_1^H$).
When $G$ is hyperbolic, $\scl$ is a {\em norm} on $B_1^H$; this is Corollary~3.57 from
\cite{Calegari_scl}, restating work of \cite{Calegari_Fujiwara} (this fact is logically superfluous for
the results of this paper).

In fact, the reader can take Lemma~\ref{euler_lemma} and Lemma~\ref{general_euler_lemma} as 
the {\em definition} of $\scl$. This is the point of view we shall take in the sequel.

\section{Free groups}\label{free_group_section}

This section contains the proof of the Rationality Theorem for free groups. 
The main result is that $\scl$ is a piecewise rational linear function on $B_1^H(F)$, for $F$
a free group. A similar statement holds for some groups derived in simple ways from free groups.
Throughout this section, we use Lemma~\ref{euler_lemma} and Lemma~\ref{general_euler_lemma} 
as an operational definition of $\scl$.

\subsection{Handlebodies and arcs}

In this section, for convenience, we use some language and basic facts from elementary
$3$-manifold topology; for a reference, see \cite{Hempel}.
In the sequel, let $F$ denote a free group of some fixed rank and let $H$ denote a handlebody
of genus equal to $\rank(F)$. 
As in Figure~\ref{handledec} (illustrating the case of $\rank(F)=4$), 
we consider a system of compressing disks $D_i$
which decompose $H$ into $\rank(F)-1$ components, each of which retracts down to one
of $\rank(F)-1$ compressing disks $E_j$. Denote the union of the $D_i$ by $\D$ and the
union of the $E_j$ by $\E$.

\begin{figure}[htpb]
\labellist
\small\hair 2pt
\pinlabel $E_1$ at 0 370
\pinlabel $E_2$ at -20 260
\pinlabel $E_3$ at 0 150
\pinlabel $D_1$ at 180 520
\pinlabel $D_2$ at -15 315
\pinlabel $D_4$ at -15 200
\pinlabel $D_3$ at 375 315
\pinlabel $D_5$ at 375 200
\pinlabel $D_6$ at 180 -15
\endlabellist
\centering
\includegraphics[scale=0.4]{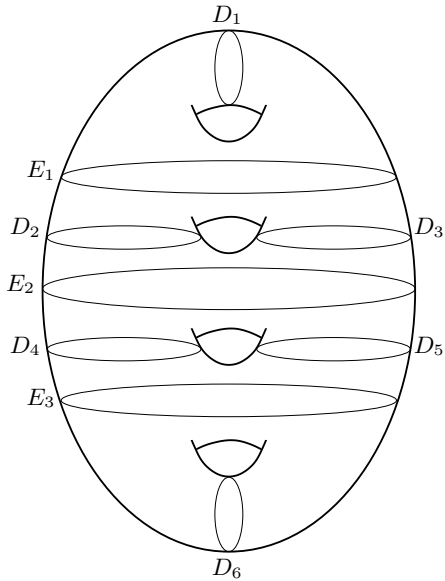}
\caption{The decomposing disks $E_i$ and $D_j$ for $g=4$} \label{handledec}
\end{figure}

Given a conjugacy class in $F$, we construct a representative loop in the corresponding
free homotopy class in $H$ of a simple kind. Such a representative
will be made up of certain kinds of arcs, which we call {\em horizontal} and {\em vertical},
and which are defined as follows.

\begin{definition}
A {\em horizontal} arc is an embedded arc $\alpha:I \to H$ whose image is contained in some $E_i$.
A {\em vertical} arc is an arc $\alpha:I \to H$ which is properly embedded in the complement
of $\E$, and which intersects some $D_j$ transversely in one point.
\end{definition}

Note that any two horizontal arcs with the same endpoints are homotopic rel. endpoints
through horizontal arcs. Moreover, any two vertical arcs whose endpoints are contained in the
same $E_i$ are properly homotopic through vertical arcs; call arcs which differ from each
other by such homotopies {\em equivalent}.

\begin{remark}\label{branch_remark}
If one does not want the psychological convenience of working in a manifold, one can substitute in place of
$H$ a union of $\rank(F)$ solid tori $H_i$, each with a marked disk $E_i$ in their boundary, and glue the tori up by
identifying the $E_i$ with a single disk $E$ by homeomorphisms. The resulting space is a manifold away from 
the disk $E$. In the case $\rank(F)=2$, the two approaches are equivalent. 
\end{remark}

Dual to the system $\D$ of compressing disks there is a graph $\Gamma$ with one vertex for
each component of $H - \D$ and one edge for each $D_i$ in $\D$. There is an isomorphism
$\pi_1(\Gamma) \cong \pi_1(H) \cong F$. The universal cover $\til{\Gamma}$ of $\Gamma$ is a tree
(for an introduction to trees in geometric group theory, see \cite{Serre_trees}, especially Chapter~1).
Every element in $F$ acts on the tree $\til{\Gamma}$ with a unique axis; this axis covers
a closed loop in $\Gamma$. Each arc in $\Gamma$ corresponds to a unique equivalence class of vertical
arc in $H$. So to each conjugacy class of element $g \in F$ is associated a (cyclically ordered)
sequence of (equivalence classes of) vertical arcs in $H$. If two consecutive vertical arcs
are on opposite sides of some $E_i$ in $\E$, they can be homotoped (rel. endpoints) 
until they have a common endpoint
in $E_i$, and their union is transverse to $E_i$ at that point. If two consecutive vertical arcs are
on the same side of some $E_i$ in $\E$, their endpoints in $E_i$ can be joined by a horizontal arc
in $E_i$. In this way, a conjugacy class in $F$ determines a loop $\gamma$ in $H$, unique up to
equivalence, made up of vertical and horizontal arcs. Say that such a $\gamma$ is in
{\em bridge position}.

\begin{figure}[htpb]
\centering
\includegraphics[scale=0.4]{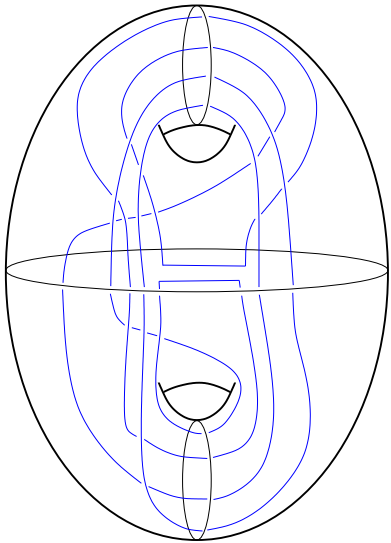}
\caption{A loop in bridge position representing the element $ababa^{-2}b^{-2}$ in $F_2$.
There are eight vertical arcs (one for each letter) and two horizontal arcs (one for
each ``double letter''). This loop happens to be embedded in $H$, but the isotopy
class of the loop is not significant, just its homotopy class.} \label{bridgeloop}
\end{figure}

If $g_1,g_2,\cdots, g_m$ is a family of elements in $F$, then a family of loops 
$\gamma_1,\cdots,\gamma_m$ in $H$ representing the conjugacy classes of the $g_i$ is
in bridge position if each loop individually is in bridge position.

See Figure~\ref{bridgeloop} for an example of a loop in bridge position.
Since each horizontal or vertical arc has distinct endpoints, a circle in bridge position decomposes
into at least two arcs. We may (and do) assume without loss of generality that any family of
circles in bridge position is actually embedded in $H$. However, it is very important to note that
the {\em isotopy class} of the loop is {\em not} important; all that matters is its {\em homotopy class}.

\begin{remark}
A system $\D$ and $\E$ of compressing disks for a handlebody determines a generating set for $F$ as a groupoid,
whose generators are equivalence classes of vertical arcs. If we use in place of a handlebody the
space described in Remark~\ref{branch_remark}, vertical arcs correspond to generators for $F$ as a group.
\end{remark}

\subsection{Polygons and rectangles}

Now, let $g \in F$ be a conjugacy class in $[F,F]$, and let $\gamma$ be a loop in bridge position
representing $g$. Let $Z \subset \E$ be the union of the endpoints of the horizontal
and vertical arcs in $\gamma$ (remember that for convenience we have assumed that $\gamma$ is embedded). 
Note that $Z$ is a finite set.
Let $f:S \to H$ be a map of a surface whose (possibly multiple) boundary components
wrap some number of times around $\gamma$. By Lemma~\ref{monotone_surface_lemma}, we may assume that
$f$ is monotone, so that each component of $\partial S$ maps over $\gamma$ with positive degree.
This is not strictly necessary for what follows, but it simplifies some arguments.
We will gradually adjust $f$ and $S$,
never changing $n(S)$ or increasing $-\chi^-$, until the end result is built up from a
finite number of simple ``pieces''.

Since by hypothesis $f$ is monotone, first adjust $f$ by a homotopy so that the restriction
of $f$ to each boundary component is an orientation-preserving covering map $\partial_i \to \gamma$. Next
put $f$ into general position (rel. $\partial S$) with respect to the
disks $\D$.

Since $f$ is in general position, and since $\gamma$ is transverse to $\D$, 
the preimage $f^{-1}(\D)$ is a union of disjoint embedded loops and proper arcs. Since $f$
restricted to $\partial S$ is an immersion, every arc of $f^{-1}(\D)$ is essential.
Let $\alpha$ be a component of $f^{-1}(\D)$ in $S$. If $\alpha$ is an (innermost)
inessential loop, then $\alpha$ can be pushed off $\D$ by a homotopy, 
reducing the number of components of $f^{-1}(\D)$. If $\alpha$ is an
essential loop, then $S$ can be compressed along $\alpha$, and the compressing disks
can be mapped to the component of $\D$ containing $f(\alpha)$. This does not change $n(S)$
but reduces $-\chi^-$. After finitely many such compressions and homotopies, we can
assume that $f^{-1}(\D)$ is a union of disjoint embedded proper essential arcs. In short, we
have proved the following ``preparation lemma'':

\begin{lemma}\label{tight_lemma}
Let $f:S \to H$ monotone be given. Then after possibly replacing $f,S$ with $f',S'$ satisfying
$-\chi^-(S') < -\chi^-(S)$ and $n(S') = n(S)$, we can assume that $f^{-1}(\D)$ is a union
of disjoint embedded proper essential arcs.
\end{lemma}

Note that distinct components of $f^{-1}(\D)$ might be (and typically will be) parallel in $S$, 
especially for complicated $\gamma$. Let $\RR$ be a regular neighborhood of $f^{-1}(\D)$ in
$S$, so that $\RR$ consists of a union of disjoint embedded proper essential rectangles. By general
position we can take $\RR$ to be equal to the inverse image (under $f^{-1}$) of an open tubular
neighborhood $N(\D)$ of $\D$. Now, there is a deformation retraction of pairs
$H - N(\D), \gamma \cap (H-N(\D))$ to $\E,\gamma \cap \E$. This deformation retraction can
be extended to a map $r:H \to H$ which restricts to a homotopy equivalence of pairs from 
$N(\D),\gamma \cap N(\D)$ to $H-\E, \gamma \cap (H-\E)$. So after composing $f$ with such a retraction
$r$ we get a new map, which by abuse of notation we also denote by $f$, homotopic to the old
map $f$, such that $\RR = f^{-1}(H-\E)$. Each component $\beta$ of $\partial S - \RR$ 
either maps by $f$ to a point in $Z$,
or to a horizontal arc in $\gamma$. In the first case, we collapse $\beta$ to a point in $S$ by
a homotopy equivalence, so we assume without loss of generality that every arc of $\partial S - \RR$
is a horizontal arc. See Figure~\ref{retract}.

\begin{figure}[htpb]
\centering
\includegraphics[scale=0.4]{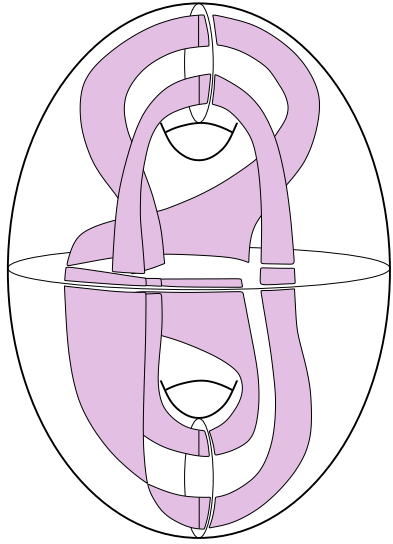}
\caption{After retracting $H-N(\D)$ to $E$, the preimage
$f^{-1}(H-E)$ consists of a union $\RR$ of rectangles.} \label{retract}
\end{figure}

Let $P$ be a component of $S - \RR$. Then $f$ maps $P$ to some component $E_i$ of $E$.
If $P$ is not a disk, then it contains an essential simple loop $\gamma$. Since
$E_i$ is a disk, $f$ maps $\gamma$ to a homotopically trivial loop in $E_i$,
so we can compress $S$ along $\gamma$, mapping the compressing disks to $E_i$,
to get a new surface $S'$ with
$-\chi^-(S') < -\chi^-(S)$ and $n(S') = n(S)$. So without loss of generality we can assume that
each component $P$ is a disk. In fact, $P$ inherits the structure of a polygon, whose edges are
arcs of the boundary of components of $\RR$, and horizontal arcs. The vertices of $P$ map by
$f$ to $Z$.

\begin{lemma}\label{embedded_vertices}
With notation as above, after possibly replacing $f,S$ with $f',S'$ satisfying
$-\chi^-(S') < -\chi^-(S)$ and $n(S') = n(S)$ we can assume that for each polygon $P$
of $S - \RR$ the image of the vertices of $P$ under $f$ are {\em distinct} elements of $Z$.
\end{lemma}
\begin{proof}
Suppose that $P$ is a polygon mapping under $f$ to the disk $E$, and let $u,v$ be distinct vertices
of $P$ mapping to the same point $y \in Z \cap E$. Let $\beta$ be an embedded proper arc in $P$
from $u$ to $v$. Let $S''$ be obtained from $S$ by identifying $u$ to $v$, then cut $S''$ along
the loop $\beta$, and glue in two disks. The restriction of $f$ to $S'' - \beta$ extends to
these disks, since $E$ is contractible. The result is a new surface $S'$ as in the statement of the
Lemma.
\end{proof}

\begin{remark}
The argument of Lemma~\ref{embedded_vertices} may be summarized by saying that $S'$ is obtained
from $S$ by first adding an oriented $1$-handle from $u$ to $v$ which maps trivially under $f$, and then compressing
the trivial embedded loop that runs over the core of this $1$-handle. Adding a $1$-handle increases $-\chi^-$ by $1$, but
doing a compression reduces it by $2$. These operations are uniquely defined up to homotopy, and keep the
surface oriented and the map monotone. Notice that $S$ and $S'$ might have different numbers of boundary
components. Call the result of this whole operation a {\em boundary compression}.
\end{remark}

\subsection{Simple branched surfaces}

The core of each component of $\RR$ runs between two vertical arcs of
$\gamma$, after mapping by $f$. There are only finitely many combinatorial types of such rectangles;
an upper bound is the number of pairs of vertical arcs in $\gamma$ which intersect the same component
of $\D$. In fact, since $S$ is monotone, the two (oriented) arcs of $\gamma$ in the boundary of each 
rectangle must run over a handle of $H$ in opposite directions.
Similarly, there are only finitely many polygon types $P$, since each is determined by
a cyclically ordered subset of $Z \cap E_i$ for some $E_i$. 

\medskip

It is convenient to introduce the language of branched surfaces in what follows. We give
cursory definitions below to standardize terminology, but the reader who is not familiar
with branched surface should consult e.g.\ \cite{Calegari_foliations} Chapter~6, \S~6.3, or 
\cite{Mosher_Oertel} \S~1 for precise definitions.

A {\em branched surface} is a finite, smooth $2$-complex obtained from a finite
collection of smooth surfaces by gluing compact subsurfaces.
The set of non-manifold points of a branched surface $B$ is called
the {\em branch locus}, and denoted $C(B)$. A branched surface in which the branch locus is a $1$-manifold is
called a {\em simple branched surface}. In this paper we are only interested in
simple branched surfaces. One can also consider branched surfaces with
boundary; the boundary of a branched surface is a {\em train-track} (see e.g.\ \cite{Harer_Penner}).
The components of $B-C(B)$ are called the {\em sectors} of the branched surface. A branched surface
is oriented if the sectors can be oriented in such a way that the orientations are compatible
along $C(B)$. In a simple branched surface, several distinct sectors locally bound
a component of the branch locus from either side; see Figure~\ref{local_models} for an
example of the local model.
\begin{figure}[htpb]
\centering
\includegraphics[scale=0.5]{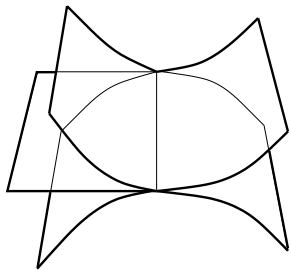}
\caption{An arc of the branch locus with two sectors on one side and three sectors
on the other} \label{local_models}
\end{figure}
A {\em weight} is a linear function $w$ from the sectors of $B$ to $\R$ satisfying the
{\em gluing conditions}: along each component of $C(B)$, the weights of the sectors on each
side sum to the same value. 

A branched surface $B$ has a well-defined tangent bundle, even along the branch locus,
so it makes sense to say that a map from a surface into $B$ is an immersion.
A proper immersion $f:S \to B$ from an oriented surface to an oriented branched surface $B$ is said to
be a {\em carrying map}, and $S$ is said to be {\em carried} by $B$. A carrying map determines
a weight $w$, where $w(\sigma)$ is the cardinality of $f^{-1}(p)$ for $p$ a point in the interior of
a sector $\sigma$. Given a non-negative integral weight $w$, we say that a carrying map $f:S \to B$
{\em realizes} $w$ if the weight associated to $S$ is $w$. Note that $S$ is not uniquely determined
by its weight in general. 

For an abstract branched surface,
not every non-negative integral weight is realized, but for a {\em simple} branched surface,
every non-negative integral weight is realized. This is a somewhat subtle point, so we make
a few clarifying remarks. Let $w$ be a weight. An edge $e$ of $C(B)$ has sectors
$\sigma_1,\cdots,\sigma_m$ on one side, and $\sigma_{m+1},\cdots,\sigma_n$ on the
other. Take $w(\sigma_i)$ copies
of $\sigma_i$ for each $i$. The gluing condition along $e$ is the equality
$$\sum_{i=1}^m w(\sigma_i) = \sum_{i=m+1}^n w(\sigma_i)$$
Hence there is a bijection between the set of copies of the $\sigma_1,\cdots,\sigma_m$ 
and the set of copies of the $\sigma_{m+1},\cdots,\sigma_n$. 
After choosing such a bijection, these copies can be glued up along copies of
$e$ to build the surface $S$ near $e$. 
For a general branched surface, there is a holonomy problem at the triple points of $C(B)$:
around each triple point, the holonomy must be trivial, or else the surface constructed will
map to $B$ with branch points. But if $B$ is simple, there are no such triple
points, and the holonomy problem goes away. This concludes our summary of the theory of branched
surfaces.

\begin{remark}
If $B$ is an {\em embedded} branched surface (possibly with triple points) in a $3$-manifold $M$,
the local transverse order structure canonically solves the holonomy problem: associated
to each weight there is a unique {\em embedded} surface contained in a tubular neighborhood of
$B$. The requirement that the surface be embedded makes the bijections along edges canonical,
and therefore the holonomy around a triple point is necessarily trivial. A similar phenomenon
occurs in normal surface theory: a vector of weights satisfying the gluing equalities
and inequalities determines a unique embedded normal surface. When one tries to do immersed
normal surface theory in $3$-manifolds, the holonomy problem reasserts itself and the
situation becomes very tricky; see \cite{Rannard} for a discussion of some of the
phenomena which arise in the theory of immersed normal surfaces. 
\end{remark}

We construct a branched surface $B$ as follows. $B$ is obtained by taking a disjoint 
copy of each combinatorial type
of (oriented) marked polygon $P$, and a copy of each combinatorial type of (oriented) rectangle $R$
as above, and gluing them along their common oriented edges. There is a unique way to choose the smooth structure along
$C(B)$ compatibly with the orientations. Notice that each marked polygon is glued
to exactly one rectangle along each non-horizontal edge, but each rectangle is typically glued
to several polygons along each of two edges. Hence a weight is determined by its values on polygon sectors
(cf.\ quadrilateral co-ordinates in normal surface theory). The resulting branched surface is special in
a few ways, of which we take note:
\begin{enumerate}
\item{It is {\em oriented}, since it is obtained by gluing oriented pieces in a
manner compatible with their orientations.}
\item{The branch locus consists of a finite union of embedded arcs along which a rectangle
is glued to several possible polygons. In particular, the branch locus contains no triple points
and is {\em simple}.}
\item{The branched surface has boundary which is an oriented train-track. Every surface $S$ carried by $B$
has boundary $\partial S$ which is carried by $\partial B$.}
\item{The branched surface admits a tautological immersion $\iota:B \to H$, by choosing a map for
each rectangle and polygon type and gluing these maps together along the edges. The restriction of
$\iota$ to each train-track component of $\partial B$ is an oriented immersion to $\gamma$}
\end{enumerate}

For each polygon $P$ let $s(P)$ denote the number of edges, and $h(P)$ the number of horizontal edges.
Note that $s(P) \ge 2$ and $h(P)\le s(P)/2$, since a pair of adjacent edges of a polygon
cannot both be horizontal.
Furthermore, if $s(P)=2$ then $h(P)=0$. Let $g:S \to H$ be decomposed into polygons and rectangles.
Each rectangle contributes $0$ to $\chi(S)$, and each polygon $P$ contributes
$$\chi(P) = \frac {-(s(P) - 2 - h(P))} 2$$
This can be seen by giving each rectangle and polygon a singular foliation tangent to $\partial B$ and
transverse to $C(B)$, and using the Hopf-Poincar\'e formula.

In particular, the contribution of each polygon to $\chi(S)$ is {\em nonpositive}, and therefore
$-\chi^-(S)$ is a {\em linear function} of the number and kinds of rectangles and polygons that make it up.

In summary, the results of the previous sections show that for any monotone $f:S \to H$, after possibly replacing
$S$ by a new surface with smaller $-\chi^-$ and the same $n(S)$, we can homotope $f$ so that
it factors through a carrying map to $B$, and determines a non-negative integral weight.
Conversely, since $B$ is simple, every non-negative integral weight on $B$ corresponds to
a surface $S$ as above (possibly not unique), and $-\chi^-$ is a linear function of the values
of $w$ on the branches of $B$. Notice that $-\chi^-$ depends only on $w$, though the
topology of $S$ might not.

\medskip

Let $W$ denote the (finite dimensional) real vector space of weights on $B$, and $W^+$ the subspace
of nonnegative weights. Moreover, let $W(\Q)$ (resp. $W^+(\Q)$ )
denote the subset of weights with rational (resp. non-negative rational) coefficients. Then
$-\chi^-$ is a linear function on $W$ which is non-negative on $W^+$, and takes rational values
on $W(\Q)$. We abbreviate this by saying that $-\chi^-$ is a {\em rational linear function}.

Let $V$ be the subspace of $W$ spanned by $W^+$ (note that this subspace is not necessarily
equal to $W$). Then $V$ is a rational subspace of $W$, since it is spanned by finitely many
rational vectors. There is a rational linear function $\partial: V \to \R$ defined as follows.
Given a positive integral weight $w$, let $S$ be a surface carried by $B$ associated to $w$.
Then define
$$\partial(w) = n(S)$$
and extend by linearity to $V$. Notice that $\partial(w)$ does not depend on the choice of $S$, but only on the
induced weight on the train-track $\partial B$.

The inverse $\partial^{-1}(1) \cap W^+$ is a rational polyhedron.
Moreover, by construction, there is an equality
$$\scl(g) = \inf_{w \in \partial^{-1}(1) \cap W^+} \frac {-\chi^-(w)} 2$$
Since $-\chi^-$ is non-negative on $W^+$, this infimum is {\em realized}, and the
set of points which realize the infimum is itself a rational polyhedron. 
If $w$ is an integral weight in the projective class of an element of this polyhedron
and $S$ is a surface carried by $B$ with weight $w$, then $S$ is an extremal surface for $g$.

\medskip

From this discussion we can conclude the following:

\begin{enumerate}
\item{$\scl(g) \in \Q$ for every $g \in [F,F]$}
\item{An extremal surface exists for every $g$}
\item{There is an algorithm to calculate $\scl$ and to construct all monotone extremal surfaces
for every $g \in [F,F]$}
\end{enumerate}

\begin{remark}
If $f:S \to H$ is extremal, and there is an essential embedded loop $\alpha$ in $S$ such that
$f(\alpha)$ is freely homotopic to
a power of $\gamma$, we may cut open $S$ along $\alpha$ to produce a new extremal (but not necessarily 
monotone) surface.
\end{remark}

\subsection{Polyhedral norm}

In fact, very little is required to extend the results of the last few sections to finite dimensional
vector spaces of $B_1(F)$. Let $g_1,g_2,\cdots,g_m \in F$ be represented by loops $\gamma_1,\gamma_2,\cdots,\gamma_m$
in bridge position in $H$. Denote $\gamma = \cup_i \gamma_i$.
Fix an orientation on each $\gamma_i$. 
A monotone surface $f:S \to H$ whose boundary wraps some positive number of times around the various $\gamma_i$ can
be compressed, boundary compressed and homotoped until it is composed of a union of rectangles and
polygons carried by a fixed simple branched surface $B$ as above. There is a rational linear map
$\partial: V \to H_1(\bigcup_i \gamma_i;\R) \cong \R^m$. 

Let $K$ be the kernel of the inclusion map
$H_1(\bigcup_i \gamma_i;\R) \to H_1(H;\R)$. Let $K^+$ be the intersection of $K$ with the orthant spanned
by non-negative combinations of the $[\gamma_i]$ in $H_1(\bigcup_i \gamma_i;\R)$. Given $k \in K^+$
corresponding to a collection of non-negative weights on the $g_i$,
we have an equality
$$\scl(k) = \inf_{w \in \partial^{-1}(k)\cap W^+} \frac {-\chi^-(w)} 2$$ 
and therefore $\scl$ is a piecewise rational linear function on $K^+$. There are finitely many orthants
of this kind, corresponding to choices of orientation on each $\gamma_i$, so $\scl$ is piecewise rational linear
on $K$.

Putting this together proves our main result:

\begin{main_rat_thm}
Let $F$ be a free group.
\begin{enumerate}
\item{$\scl(g) \in \Q$ for all $g \in [F,F]$.}\label{rational_bullet}
\item{Every $g  \in [F,F]$ rationally bounds an extremal surface (in fact, every rational chain $C$ in $B_1^H$ rationally
bounds an extremal surface)}\label{extremal_bullet}
\item{The function $\scl$ is piecewise rational linear on $B_1^H$.}\label{polyhedron_bullet}
\item{There is an algorithm to calculate $\scl$ on any finite dimensional rational subspace of
$B_1^H$.}
\end{enumerate}
\end{main_rat_thm}

Note that bullet~(\ref{rational_bullet}) is a special case of bullet~(\ref{polyhedron_bullet}).

\subsection{Other groups}

\begin{definition}
Say that a group $G$ is PQL (pronounced ``pickle'')
if $\scl$ is piecewise rational linear on $B_1^H(G)$.
\end{definition}

The PQL property is inherited by supergroups of finite index. This follows in a straightforward way 
from the following Lemma, which relates $\scl$ in groups and in finite index subgroups.

\begin{lemma}\label{finite_index_formula}
Let $G$ be a group, and $H$ a subgroup of $G$ of finite index. Let $X$ be a CW complex with
$\pi_1(X) = G$, and let $\widehat{X}$ be a covering space with $\pi_1(\widehat{X})=H$. Let
$g_1,\cdots,g_m$ be elements in $G$, and for each $i$, let 
$\gamma_i$ be a loop in $X$ representing the conjugacy
class of $g_i$. Let $\beta_1,\cdots,\beta_l$ be the preimages of the
$\gamma_i$ in $\widehat{X}$, and
$h_1,\cdots,h_l$ the corresponding conjugacy classes in $H$. Then
$$\scl_H(\sum h_i) = |G:H| \cdot \scl_G(\sum g_i)$$
\end{lemma}
\begin{proof}
If $f:S \to X$ is a map of a surface whose boundary maps to the $\gamma_i$, then $S$ admits a finite
index cover $\widehat{S}$ such that $f$ lifts to $\widehat{f}:\widehat{S} \to \widehat{X}$.
Conversely, given $f:S \to \widehat{X}$ with boundary mapping to the $\beta_i$, the
composition of $f$ with the covering projection $\widehat{X} \to X$ takes $\partial S$ to $\gamma$.
Now apply Lemma~\ref{euler_lemma} and Lemma~\ref{general_euler_lemma}.
\end{proof}

\begin{theorem}
Let $M$ be a non-compact Seifert-fibered $3$-manifold. Then $\pi_1(M)$ has the PQL property.
\end{theorem}
\begin{proof}
Since $M$ is noncompact, there is a finite index subgroup $H$ of $\pi_1(M)$ of the form
$H = \Z \oplus F$ where $F$ is free. Since $\Z$ is amenable, a theorem of Bouarich (see
\cite{Bouarich} or \cite{Calegari_scl}, \S~2.4) implies that $H^2_b(F) \to H^2_b(H)$ is
an isomorphism, and $Q(H)=Q(F)\oplus H^1(\Z)$. Hence $H$ is PQL by
Theorem~\ref{general_duality_theorem}. But then $\pi_1(M)$ is PQL by Lemma~\ref{finite_index_formula}.
\end{proof}

\begin{example}
The braid group $B_3$ is isomorphic to $\pi_1(S^3 - K)$ where $K$ is the trefoil knot
(of either handedness). This group is a central $\Z$ extension of $\Z/2\Z * \Z/3\Z$, which
admits a free subgroup of finite index. Hence $B_3$ has the PQL property.
\end{example}

\begin{question}
Does every $3$-manifold group have the PQL property?
\end{question}

\section{Computing stable commutator length}\label{algorithm_section}

In this section we discuss the implementation of the algorithm implicit in \S~\ref{free_group_section}, 
and study an explicit example. In order to describe the algorithm in a uniform way for free groups of any rank,
it is convenient to use the formalism described in Remark~\ref{branch_remark}, in which vertical arcs
in $\Gamma$ correspond to elements in a free generating set for $F$. If one is more comfortable working
with systems of compressing disks in a handlebody, one must work with {\em groupoid} generators for $F$ associated to
the system of vertical arcs in a splitting. In the case of a free group of rank $2$, both formalisms agree; the cautious
reader may prefer to stick to this case in what follows.

\medskip

Fix a free symmetric generating set $S$ for a free group $F$, 
and let $C= w_1 + w_2 + \cdots + w_m$ be an integral chain in $B_1^H(F)$, expressed as a formal sum of
cyclically reduced words in the generating set.
We denote the generators of $F$ by $a,b,c,\cdots$ and their inverses by $A,B,C,\cdots$.

\begin{definition}\label{arcs_and_polygons}
A {\em letter} is a specific character in a specific word; its {\em value} is the element of the generating set
that it represents.
An {\em arc} is an ordered pair $(\ell_1,\ell_2)$ of letters whose values are inverse in $F$. A {\em polygon} is
a cyclically ordered list of distinct
arcs $\alpha_0,\alpha_1,\cdots,\alpha_{m-1}$ so that for each $i$, the last letter of
$\alpha_i$ immediately precedes the first letter of $\alpha_{i+1}$ in some (cyclic) word $w_j$ (indices $i$ and
$i+1$ are taken mod $m$). The {\em length} of a polygon is the number of arcs appearing in the list.
\end{definition}

In the language of \S~\ref{free_group_section}, each oriented rectangle has four edges, two of which 
(the ``free edges'') correspond to vertical edges in $\Gamma$, and two of which (the ``glued edges'')
are glued up to some polygon. Each vertical edge in $\Gamma$ corresponds to a letter in some $w_i$. The two glued
edges of a rectangle determine two arcs $\alpha$ and $\alpha'$ consisting of the same pair of letters in opposite orders;
i.e.\ if $\alpha = (\ell_1,\ell_2)$ then $\alpha' = (\ell_2,\ell_1)$.

A polygon (in the sense of \S~\ref{free_group_section}) has edges that are glued up to edges of rectangles,
and edges on horizontal edges of $\Gamma$. Horizontal edges are of purely psychological value; they do not contribute
to $\chi$. So to each such polygon there is associated a polygon (in the sense of Definition~\ref{arcs_and_polygons}) 
consisting of a cyclically ordered list of arcs. The condition that the arcs making up a polygon are distinct is the analog
of the condition that each polygon (in the sense of \S~\ref{free_group_section}) has vertices mapping to distinct
points of $Z$.

Let $P$ denote the real vector space spanned by the set of polygons. Let $A$ denote the vector space spanned
by the set of arcs modulo the relation that $(\ell_1,\ell_2) = -(\ell_2,\ell_1)$. There is a linear map
$\partial:P\to A$ sending a polygon to the formal sum of the arcs that make it up. Let $W$ denote the kernel of $\partial$,
and $W^+$ the cone of non-negative vectors in $W$. The relations $(\ell_1,\ell_2)= -(\ell_2,\ell_1)$ and the condition
$\partial=0$ express the gluing equations for $B$ in terms of weights on polygon sectors.
The spaces $W$ and $W^+$ are naturally isomorphic to the spaces with the
same names described in \S~\ref{free_group_section}. A polygon of length $m$ contributes $(m-2)/4$ to $-\chi^-/2$.
The condition that a formal sum of polygons has boundary equal to the chain $C$ is a further list of linear conditions, one for
each word in $C$. Minimizing the objective function $-\chi^-/2$ is a linear programming problem, which can
be solved in a number of ways.

\begin{example}
Let $w=ababABaBAbAB$ (remember that $A=a^{-1}$ and $B=b^{-1}$). There are $36$ arcs, one for each ordered pair of
letters in $w$ with opposite values. Labeling the
letters of $w$ by integers from $0$ to $11$, and using the shorthand $X=10$ and $Y=11$, these arcs are
$$04, 40, 08, 80, 0X, X0, 15, 51, 17, 71, 1Y, Y1, 24, 42, 28, 82, 2X, X2,$$
$$35, 53, 37, 73, 3Y, Y3, 46, 64, 59, 95, 68, 86, 6X, X6, 79, 97, 9Y, Y9$$
There are 625 polygons, 43 of length 4 or less (since the letters of $w$ alternate between one of $a^\pm$
and one of $b^\pm$, every polygon has even length), including
$$(04, 51, 28, 9Y), (04, 53, 40, 1Y), \cdots, (X6, 79)$$
Let $P=\R^{625}$ denote the vector space of formal linear combinations of polygons, and let $p_i$
for $0 \le i \le 624$ denote the components of a vector in $P$. For each $i$, let $l_i$ denote the length of
polygon $i$. Compatibility of gluing along rectangles (i.e.\ $\partial = 0$ above) imposes one equation of
the form $\sum p_k = \sum p_l$ for each pair of arcs $ij, ji$ where the polygons of type $k$ are those that
contain the arc $ij$, and polygons of type $l$ are those that contain the arc $ji$ (note that a polygon
type might contain both $ij$ and $ji$ or neither). There are half as many equations of this kind as arcs,
hence $18$ equations. 

Restricting to geometrically sensible answers imposes the conditions $p_i \ge 0$.
The condition that the boundary of a (formal) surface corresponding to a weight represents
$[\gamma]$ in homology is the equation
$$\sum_i l_i p_i = |w| = 12$$
Subject to this list of constraints, which determine a compact convex polyhedron in $\R^{625}$, we minimize
the objective function
$$\frac {-\chi^-}{2} = \sum_i \frac {(l_i -2)p_i} 4$$

This linear programming problem can be solved using exact arithmetic by the GNU package
{\tt glpsol} \cite{glpsol} or Masashi Kiyomi's program {\tt exlp} \cite{exlp} using
Dantzig's simplex method (see \cite{Dantzig}), and gives
the answer $\scl(w) = 5/6$. 

\medskip

An extremal solution found by the simplex method is always a vertex,
which can be projectively represented by an extremal surface. One such extremal surface found by this method is
determined by the identity
\begin{multline*}
[abaB, ABAbaBabAbaBABAbaBababABB] \cdot [ABAba, BabAbaBABAbba] \\
\cdot [BabABababA, aaBAAb] = a(baBABAbaBabA)^3A
\end{multline*}
exhibiting the cube of $baBABAbaBabA$ (i.e.\ the inverse of a cyclic conjugate of
$ababABaBAbAB$) as a product of three commutators. 
Since extremal surfaces are $\pi_1$-injective,
the group generated by $abaB,\cdots, aaBAAb$ is a free subgroup
of $F_2$ of rank $6$, equal to the image of $\pi_1$ of a genus $3$ surface with one puncture under
an injective homomorphism. This fact may be verified independently e.g.\ by using Stallings folding,
and gives an independent check of the validity of the calculation.
\end{example}

As was discussed in the introduction, this example answers Bavard's question in the negative. Other
simple examples are
$$\scl(aababaBAbAABAB) = 2/3$$ 
and
$$\scl(a + BB + AAbab) = 3/4$$

\subsection{Addendum}

Several developments building on this material have taken place
between the time this paper was submitted and was accepted for publication. It would be inappropriate to
go into too much detail, but for the convenience of the reader, we summarize some of the most interesting
points.

Firstly, the algorithm described above has been improved to run in {\em polynomial time}. An account of
this improvement is described in \cite{Calegari_scl}, \S~4.1.7--8. The program {\tt scallop} (source code
available at \cite{Calegari_scallop}) implements this algorithm, and can be used to compute $\scl$ in
$F_2$ on words of length $\sim$ 60. Experiments using {\tt scallop} reveal additional structure in the $\scl$ spectrum
of a free group, partially explained in a forthcoming paper \cite{Calegari_sails}. 
In particular, it is possible
to prove rigorously that in any nonabelian free group, the image of $\scl$ contains nontrivial accumulation
points, and takes on values with any denominator.

Secondly, the polyhedral structure on the $\scl$ norm unit ball is studied in \cite{Calegari_faces}, and it is
shown that every realization of a free group $F$ as $\pi_1(S)$ where $S$ is an oriented surface with boundary
is associated to a codimension one face of the boundary of the $\scl$ norm unit ball in $B_1^H$.

Thirdly, the extremal surfaces guaranteed in bullet~(\ref{extremal_bullet}) of the Rationality Theorem are
exploited to construct closed surface subgroups in certain graphs of free groups amalgamated along 
cyclic subgroups. This is studied in \cite{Calegari_surface}, and further generalized in \cite{Gordon_Wilton}.

\section{Acknowledgment}

While writing this paper I was partially funded by NSF grant DMS 0707130.
I would like to thank Roger Alperin, Nathan Dunfield, Dieter Kotschick, Lars Louder, 
Jason Manning, Bill Thurston, Dongping Zhuang, the anonymous referees, and the members of the d\=osemi 
at the Tokyo Institute of Technology, especially Shigenori Matsumoto. I presented an 
incorrect argument, purporting to prove the main result of this paper, at the d\=osemi
in April 2007. Matsumoto asked a simple-sounding question about branch points, which turned
out to be a crucial detail that I had overlooked. Understanding this detail was the key
to obtaining the results in this paper. I would also like to explicitly thank Jason
Manning for a number of conversations that substantially increased my
confidence in the experimental results discussed in \S~\ref{algorithm_section}.

\end{document}